\documentclass[12pt]{article}
\usepackage[utf8]{inputenc}
\usepackage{chapterbib}
\usepackage[sectionbib,numbers]{natbib}
\usepackage{eurosym}
\usepackage{amstext,amsthm}
\usepackage{amsmath}
\usepackage{amssymb,mathtools}
\usepackage[nointegrals]{wasysym} 
\usepackage{mathrsfs}
\usepackage{hyperref}
\usepackage{graphicx}

\newtheorem{proposition}{Proposition}[section]
\newtheorem{corollary}[proposition]{Corollary}
\newtheorem{theorem}{Theorem}
\newtheorem{lemma}[proposition]{Lemma}
\theoremstyle{definition}
\newtheorem{definition}[proposition]{Definition}
\newtheorem{remark}[proposition]{Remark}
\newtheorem{assumption}[proposition]{Assumption}

\usepackage[normalem]{ulem}
\usepackage{color}

\DeclareMathAlphabet{\mathpzc}{OT1}{pzc}{m}{it}

\newcommand{\bE}{\mathbb{E}}

\usepackage{times}
\numberwithin{equation}{section}

\newcommand\unnumberedfootnote[1]{ %
        \let\temp=\thefootnote %
        \renewcommand{\thefootnote}{}%
        \footnote{#1}%
        \let\thefootnote=\temp%
        \addtocounter{footnote}{-1}}

\begin{document}

\setcounter{page}{1}
\title{\LARGE Mean-field limits for non-linear Hawkes processes with excitation and inhibition}

\author{{\sc by  P. Pfaffelhuber, S. Rotter and J. Stiefel} \\[2ex]
  \emph{Albert-Ludwigs University Freiburg} } \date{\today}

\maketitle

\unnumberedfootnote{\emph{AMS 2000 subject classification.} {\tt 60G55}
  (Primary) {\tt, 60F05} (Secondary).}

\unnumberedfootnote{\emph{Keywords and phrases.} Multivariate Hawkes process; Volterra equation; spike train}

\begin{abstract}
  \noindent
  We study a multivariate, non-linear Hawkes process $Z^N$ on the complete graph with $N$ nodes. Each vertex is either excitatory (probability $p$) or inhibitory (probability $1-p$). We take the mean-field limit of $Z^N$, leading to a multivariate point process $\bar Z$. If $p\neq\tfrac12$, we rescale the interaction intensity by $N$ and find that the limit intensity process solves a deterministic convolution equation and all components of $\bar Z$ are independent. In the critical case, $p=\tfrac12$, we rescale by $N^{1/2}$ and obtain a limit intensity, which solves a stochastic convolution equation and all components of $\bar Z$ are conditionally independent.
\end{abstract}

\section{Introduction}
In \cite{Hawkes1971a, HawkesOakes1974}, Hawkes processes were introduced as self-excitatory point processes. Today, they are used in
various fields of applications including seismology \cite{Ogata1988,
  Fox2016}, interactions in social networks \cite{Zipkin2016,
  lukasik-etal-2016-hawkes}, finance \cite{Bacry2015, Hawkes2018} and
neuroscience \cite{pmid28234899, pmid25974542}. In the classical univariate,
linear Hawkes process, the firing rate at time $t$ is a linear
function of $I_t := \sum_i \varphi(t-T_i)$, where the $T_i$'s are
previous jump times.  In this case, since rates cannot become
negative, $\varphi\geq 0$ is required, leading to a self-excitatory process.

Our main motivation to study Hawkes processes comes from the neurosciences. In a graph, vertices model neurons, whereas the (directed) edges are synapses linking the neurons. A point process indexed by the vertices models action potentials or spike trains of electrical impulses. Communication via synapses leads to correlated point processes such that each spike in one neuron influences the rate by which a neighboring vertex fires. In this field, it is known that neurons cannot only excite others, but inhibition is another important factor (see e.g.\ \cite{Kandel-2012-Book}). The main goal of this paper is to obtain limit results for the multivariante, non-linear Hawkes process, where the firing rate at time $t$ is $h(I_t)$ and $\varphi \leq 0$ can occur as well, which we
interpret as inhibition.

Nonlinear Hawkes processes have been studied to some extent in the past decades. \cite{BremaudMassoulie1996,Massoulie1998} focus on ergodic properties using Lipschitz conditions of the transfer function $h$.  In \cite{Zhu2013, Zhu2015}, central limit theorems and large deviation results for a univariate nonlinear Hawkes process are given. The case of inhibition in a multivariate setting was studied in \cite{Chen2019a} using a thinning process representation, while \cite{Costa2020} uses renewal techniques to establish limit theorems in the univariate setting.

Mean-field models have frequently been applied in the life sciences; see e.g.\ \cite{DaiPra2017} where particular applications in neuroscience are discussed. Mean-field limits of nonlinear Hawkes processes were first studied by \cite{delattre2016hawkes}. This has been extended to age-dependent nonlinear Hawkes processes by \cite{Chevallier2017a, Chevallier2017b}, and to specific models including multiple classes \cite{DitlevsenLoecherbach2017}. In addition to the law of large numbers in the mean-field model, a central limit theorem is proved in {\color{black} \cite{Chevallier2017b, HeesenStannat2020}}.  Another branch of research studies a mean-field limit of a spatially extended (geometric) Hawkes process, showing a law of large numbers \cite{Chevallier2019} and central limit theorem \cite{Chevallier2020}. Here, a first proof of the neural field equation could be obtained.

~

In the present paper, our approach is a combination and extension of {\color{black} \cite{Chevallier2017a},\cite{HeesenStannat2020} and \cite{erny2020mean}}, with a special focus on models including excitatory and inhibitory neurons/vertices. We assume that each vertex/neuron is either excitatory or inhibitory, i.e.\ excites or inhibits all of its neighbors. We will denote by $p$ the fraction of excitatory vertices/neurons and distinguish the critical case $p=1/2$ from the non-critical one. For the latter, we obtain in Theorem~\ref{T1} a classical mean-field result, i.e.\, by rescaling the interaction intensity by $N$, a deterministic limit of the intensity and independent point processes driven by this intensity arises. We also provide a central limit result for the intensity. In Theorem~\ref{T2}, we are dealing with the critical case. Here, we rescale the interaction intensity by $N^{1/2}$, leading to a limiting intensity which is itself stochastic, and drives conditionally independent point processes. {\color{black} We will discuss connections of our findings to previous work in Remarks~\ref{rem:T1} and~\ref{rem:T2a}}.

~

The mean-field model we discuss here is unrealistic for applications in neuroscience for various reasons. Above all, not all neurons are connected. Sparse graphs are one possibility for a more realistic model \cite{pmid10809012}, while multiple classes of differently interconnected neurons are another option \cite{DitlevsenLoecherbach2017}. For a detailed discussion on implications of our findings for neuroscience, see Section~\ref{S:discuss}.

\section{Model and main results}
We use the following general model for a non-linear Hawkes process:

\begin{definition}[Multi-variate, non-linear Hawkes process]
  \label{def:hawkes}
  Let $\mathbb G = (\mathbb V, \mathbb E)$ be some finite, directed graph, and write $ji \in \mathbb E$ if $j\to i$ is an edge in $\mathbb G$. 
  Consider a family of measurable, real-valued functions
  $(\varphi_{ji})_{ji \in \mathbb E}$ and a family of real-valued, non-negative functions
  $(h_i)_{i\in \mathbb V}$.  Then, a  point process $Z = (Z^i)_{i\in\mathbb V}$ (with state space $\mathbb N_0^{\mathbb V}$) is a multi-variate
  non-linear Hawkes process with interaction kernels $(\varphi_{ji})_{ji \in \mathbb E}$ and transfer functions $(h_i)_{i\in\mathbb V}$, if $Z^i, Z^j$ do not jump simultaneously for $i\neq j$, almost surely, and the compensator of $Z^i$ has the form $(\int_0^t \lambda_s^i ds)_{t\geq 0}$ with
  $$ \lambda_t^i := \lambda_t^i(Z_{s<t})
  := h_i\Big(\sum_{j: ji \in \mathbb E} \int_0^{t-}
  \varphi_{ji}(t-s)dZ_s^j\Big), \qquad i\in\mathbb V.$$ 
\end{definition}

\noindent
We need some mimimal conditions such that the multivariate, non-linear Hawkes process is well-defined (i.e.\ exists). As mentioned in \cite{delattre2016hawkes}, Remark~5, the law of the non-linear Hawkes process is well-defined, provided that the following assumption holds. 

\begin{assumption}\label{ass:basic}
  All interaction kernels $\varphi_{ji}, ji \in \mathbb E$ are locally integrable, and all transfer functions $(h_i)_{i\in\mathbb V}$ are Lipschitz continuous. 
\end{assumption}

\begin{remark}[Interpretation and initial condition]
\begin{enumerate}
\item   If $dZ_s^j = 1$, we call   $\varphi_{ji}(t-s) dZ_s^j$ the influence of the point at time $s$ in vertex $j$ on vertex $i$.
\item Consider the case of monotonically increasing  transfer functions. If $\varphi_{ji} \leq 0$, we then say that vertex $j$ inhibits $i$, since any point in $Z^j$ decreases the jump rate of $Z^i$. Otherwise, if $\varphi_{ji}\geq 0$, we say that $j$ excites $i$.
\item   In our formulation, we have $Z_0^i = 0, i\in\mathbb V$,     with the consequence that the $dZ_u^j$-integral in \eqref{eq:Z_ref} could also be extended to $-\infty$ without any     change. We note that it would also be possible to use some     initial condition, i.e.\ some (fixed) $(Z_t^i)_{i\in \mathbb     V, t\leq 0}$, and extend the integral to the negative reals. 
\end{enumerate}
\end{remark}

~

\noindent
Let us now come to the mean-field model, where we now fix some basic assumptions. Note that we will show convergence for large graphs, i.e.\ all processes come with a scaling parameter $N$, which determines the size of the graph.

\begin{assumption}[Mean-field setting]\label{ass1}
  Let
  \begin{enumerate}
  \item $\mathbb G_N = (\mathbb V_N, \mathbb E_N)$ be the complete graph on $N$ vertices, i.e.\ $\mathbb V_N = \{1,...,N\}$  and $ji, ij \in\mathbb E$ for all $i,j \in \mathbb V$;
  \item $h_i = h$ for all $i\in\mathbb V_N$ where $h \geq 0$ is
    bounded, $h$ and $\sqrt{h}$ are Lipschitz with constant $h_{Lip}$;
  \item $\varphi_{ji} = \theta_N U_j \varphi$ for all
    $j,i \in \mathbb V_N$, where $U_1, U_2,...$ are iid with
    $\mathbb P(U_1 = 1) = 1-\mathbb P(U_1 = -1) = p$,
    $\varphi \in \mathcal C_b^1([0,\infty))$, the set of bounded
    continuously differentiable functions with bounded derivative, and
    $\theta_N \in \mathbb R$.
  \end{enumerate}
\end{assumption}

\noindent
The form of $\varphi_{ji}$ implies that node~$j$ is exciting all other nodes with probability $p$, and inhibiting all other nodes with
probability $1-p$. By the law of large numbers, we have that
\begin{align}\label{eq:LLN}
  & \frac{1}{N} \sum_{j=1}^N U_j \xrightarrow{N\to\infty} 2p-1
    \intertext{almost surely, and by the
    central limit theorem,}\label{eq:CLT}
  & \frac{1}{\sqrt{N}} \sum_{j=1}^N (U_j+1-2p) =: W_N \xRightarrow{N\to\infty} W \sim N(0,4p(1-p)),
\end{align}
where $\Rightarrow$ denotes weak convergence. For convergence, this suggests to use $\theta_N = \tfrac 1N$ for $p\neq \tfrac 12$ and
$\theta_N = \tfrac{1}{\sqrt{N}}$ for $p=\tfrac 12$.

\section{Results on the mean-field model}
Our main goal is to give a limit result on the family $Z^N$, the multivariate, non-linear Hawkes process on the graph $\mathbb G_N$ with interaction kernels and transfer functions as given in Assumption~\ref{ass1}. We distinguish the cases $p\neq \tfrac 12$ (Theorem~\ref{T1}) and $p=\tfrac 12$ (Theorem~\ref{T2}). In both cases, the limit compensator is given by $(\int_0^t h(I_s)ds)_{t\geq 0}$, where $I$ is the weak limit of $I^N$ (see \eqref{eq:IN} and \eqref{eq:IN2}), and we find the following limit results: For $p\neq \tfrac 12$, we have that $(I_t)_{t\geq 0}$ follows a linear, deterministic convolution equation, and all components of the limit of $Z^N$ are independent. For $p=\tfrac 12$, the critical case, the process $(I_t)_{t\geq 0}$ is the unique solution of a stochastic convolution equation and all components of $Z^N$ are conditionally independent given $I$. Below, we denote by $\Rightarrow$ weak convergence in $\mathcal D_{\mathbb R^n}([0,\infty))$, the space of cadlag paths, which is equipped with the Skorohod topology; see e.g.\ Chapter~3 in \cite{EthierKurtz86}. The proof of the following result can be found in Section~\ref{ss:proofT1}.

\begin{theorem}[Mean-field limit of multi-variate non-linear Hawkes processes, $p\neq \tfrac 12$] \label{T1} Let Assumption~\ref{ass1} hold with $p\neq \tfrac12$ and $\theta_N = \tfrac 1N$. Let $Z^N = (Z^{N,1}, ..., Z^{N,N})$ be the multivariate, non-linear Hawkes process from Definition~\ref{def:hawkes}, and
\begin{align}\label{eq:IN}
    I^N(t) := \frac 1N \sum_{j=1}^N\int_0^{t-} U_j \varphi(t-s) dZ_s^{j,N}.
\end{align}
\begin{enumerate}
  \item Then, $I^N \xrightarrow{N\to\infty} I$, uniformly on compact
    time intervals in $L^2$, where $I = (I_t)_{t\geq 0}$ is the unique
    solution of the integral equation
    \begin{align}
      \label{eq:IE}
      I_t = (2p-1) \int_0^t \varphi(t-s) h(I_s)ds.
    \end{align}
  \item For all $n=1,2,...$,
    $(Z^{N,1},...,Z^{N,n}) \xRightarrow{N\to\infty} (\bar Z^1,...,\bar
    Z^n),$ 
    where $\bar Z^1, ..., \bar Z^n$ are independent and $\bar Z^i$ is a simple point process with intensity at time $t$
    given by $h(I_t)$, $i=1,...,n$. It is possible to build
    $\bar Z^1, ..., \bar Z^n$, such that the convergence is almost surely (in Skorohod-distance).
  \item Assume that $h \in \mathcal C^1(\mathbb R)$ and that $h'$ is bounded and Lipschitz. Then
    $\sqrt{N}(I^N - I) \xRightarrow{N\to\infty} K$, where $K$ is the
    unique (strong) solution of
    \begin{equation}
      \label{eq:SDE_K}
      \begin{aligned}
      K_t & = \int_0^t \varphi(t-s)  dG_s,
      \\
      G_t & = \int_0^t(Wh(I_s) + (2p-1) h'(I_s) K_s)ds
      + \int_0^t \sqrt{h(I_s)} dB_s,
            \end{aligned}
    \end{equation}
    where $W$ is given by \eqref{eq:CLT} and $B$ is a Brownian motion independent from $W$.
  \end{enumerate}
\end{theorem}

\begin{remark}
  The form of \eqref{eq:IE} tells us that $I$ follows a linear
  Volterra convolution equation \cite{berger1980volterra}. Turning into a
  differential equation, we write, using Fubini,
  \begin{align*}
    \frac{dI}{dt} & = \frac{d}{dt}(2p-1) \Big(\int_0^t \int_s^t \varphi'(r-s) h(I_s) dr ds + \int_0^t \varphi(0) h(I_s)ds\Big)
    \\ & 
         = (2p-1) \Big(\int_0^t \varphi'(t-s) h(I_s) ds + \varphi(0) h(I_t) \Big).
  \end{align*}
In particular, the special choice of $\varphi(s) = e^{-\lambda s}$ gives

  \begin{align*}
    \frac{dI}{dt} & ={\color{black} -\lambda (2p-1) \int_0^t \varphi(t-s) h(I_s) ds + (2p-1) h(I_t)  = -\lambda I_t + (2p-1)h(I_t),} 
  \end{align*}
i.e.\ $I$ follows some ordinary differential equation in this case.
\end{remark}

\noindent
While Theorem~\ref{T1} is concerned with convergence of the limit intensity of the multivariate, non-linear Hawkes process, we are also in the situation to study convergence of the average intensity of $Z^N$. The proof of the next corollary is found in Section~\ref{ss:proofCors}.

\begin{corollary}\label{cor1}
Let $Z^N = (Z^{N,1},...,Z^{N,N})$ be as in Theorem~\ref{T1}, and $\bar Z = (\bar Z^1, \bar Z^2,...)$ be as in Theorem~\ref{T1}.2. Then, 
\begin{align}\label{eq:cor11}
    \frac 1N \sum_{j=1}^N \Big(Z^{N,j} - \int_0^. h(I^N_s)ds\Big) \xrightarrow{N\to\infty} 0
    \intertext{and}\label{eq:cor12}
    \frac 1N \sum_{j=1}^N \Big(Z^{N,j} - \bar Z^j\Big) \xrightarrow{N\to\infty} 0
\end{align}
in probability, uniformly on compact intervals. Moreover, 
\begin{align}\label{eq:cor13}
    \frac{1}{\sqrt N} \sum_{j=1}^N \Big(Z^{N,j} - \int_0^. h(I_s^N)ds\Big) \xRightarrow{N\to\infty} \int_0^. \sqrt{h(I_s)}d\widetilde B_s
\end{align}
for some Brownian motion $\widetilde B$, and  (writing $h(I) := (h(I_t))_{t\geq 0}$)
\begin{align}\label{eq:cor14}
    \sqrt{N} \big( h(I^N) - h(I)\big) \xRightarrow{N\to\infty} h'(I) K.
\end{align}
Let $B, \widetilde B$ be correlated Brownian motions with $\mathbf E[B_t \widetilde B_t] = (2p-1)t, t\geq 0$. Use $B$ in the definition of $K$ in \eqref{eq:SDE_K}. Then, 
\begin{align}\label{eq:38}
    \frac{1}{\sqrt N} \sum_{j=1}^N \Big(Z^{N,j} - \int_0^. h(I_s)ds\Big) \xRightarrow{N\to\infty} \int_0^. h'(I_s)K_s ds + \int_0^. \sqrt{h(I_s)}d\widetilde B_s.
\end{align}
\end{corollary}

{\color{black}
\begin{remark}[Correlation between $B$ and $\widetilde B$] Let us briefly discuss the correlated Brownian motions appearing in \eqref{eq:38}. Clearly, the left hand sides of \eqref{eq:cor13} and \eqref{eq:cor14} sum to the left hand side of \eqref{eq:38}. The limits  $\int_0^. \sqrt{h(I_s)}d\widetilde B_s$ and $K$ appearing on the right hand sides of \eqref{eq:cor13} and \eqref{eq:cor14} are weak limits of sums of compensated point processes. While in \eqref{eq:cor13}, we sum over all point processes in the system, $I^N$ in \eqref{eq:cor14} distinguishes between nodes with different signs $U_j$.  Hence, the correlation is positive for the proportion $p$ of point processes with positive sign, and negative for the proportion $1-p$ of point processes with negative sign, summing to $p - (1-p) = 2p-1$. For more details, see the proof in Section~\ref{ss:proofCors}.
\end{remark}
}

\begin{remark}[Connections of Theorem~\ref{T1} to previous work]
\label{rem:T1} \mbox{}
\\
{\color{black} \textbf{Connections of Theorem 1.1 and Theorem 1.2 to \cite{Chevallier2017a, HeesenStannat2020}}. First note that 1. and 2. of Theorem~\ref{T1} are a particular case of \cite{Chevallier2017a}, Theorem 4.1. In their notation, we have choosen $H_{ij}=U_jh$ as interaction functions, which satisfy their condition (9), as stated in their remark 2.1 on the synaptic weights. We reformulate the result to introduce 3. of of Theorem~\ref{T1}, which does not appear in \cite{Chevallier2017a}, and to allow for a comparison to the result and proof of our Theorem~\ref{T2}.} 

Note further that Theorem~\ref{T1}.1 implies that the compensator $(\Lambda_t)_{t\geq 0}$ of each of the independent limiting processes $(\bar Z^i_t)_{t\geq 0}$ is
$$ \Lambda_t = \int_0^t h(I_s) ds= \int_0^t h\Big((2p-1) \int_0^s \varphi(s-r) h(I_r) dr\Big) ds= \int_0^t h\Big((2p-1) \int_0^s \varphi(s-r) d\Lambda_r\Big) ds.$$
See also (8) of\cite{delattre2016hawkes} and (5), (7), (8) of \cite{HeesenStannat2020} for the connection to \eqref{eq:IE}, both for the special case $p=1$ and $W=0$. In contrast to our setting, recall that \cite{HeesenStannat2020} require $\varphi(0)=0$. 
\\
\noindent
{\color{black} \textbf{Connections of Theorem 1.3 to \cite{Chevallier2017b}, \cite{HeesenStannat2020}}. Equation (5.10) from Theorem 5.6 in \cite{Chevallier2017b} gives an expression for the fluctuation of age dependent Hawkes processes. In the case where the transfer function becomes independent of the age, $\Psi(s,x)\equiv\Psi(x)$ in their notation, we discover our equation \eqref{eq:SDE_K} in the special case $p=1$ and $W=0$. As Corollary 2.2 of \cite{HeesenStannat2020} is a special case of this Theorem 5.6, (33) in that corollary is precisely \eqref{eq:SDE_K}, and the convergence of that Corollary for the special case $p=1$ (as well as $W=0$ and $\varphi(0)=0$) coincides with the statement of Theorem~\ref{T1}.3.}
\\
\noindent 
{\color{black} \textbf{Connections of Corollary~\ref{cor1} to \cite{HeesenStannat2020}}. The convergence in} \eqref{eq:cor14} generalizes Theorem~2.4 of \cite{HeesenStannat2020} to the case $p\neq 1$, $W\neq 0$ and $\varphi(0) \neq 0$.
{\color{black}Moreover, the convergence in \eqref{eq:38} generalizes Theorem~2.1 of \cite{HeesenStannat2020}: the right hand side of \eqref{eq:38} for $p=1$, $W=0$ has $B = \widetilde B$ and hence equals $G$ from \eqref{eq:SDE_K}. Using integration by parts and $\varphi(0)=0$ we obtain}
\begin{align*}
    0 & = \varphi(0) G_t - \varphi(t) G_0 = \int_0^t \varphi(t-s) dG_s - \int_0^t G_s \varphi'(t-s) ds.
\end{align*}
Therefore we can reshape $G$ from \eqref{eq:SDE_K},
\begin{align*}
    G_t & = \int_0^t h'(I_u) \int_0^u \varphi(u-s) dG_s du + \int_0^t \sqrt{h(I_s)} dB_s
    \\ & = \int_0^t h'(I_u)\int_0^u \varphi'(u-s)G_s dsdu + \int_0^t \sqrt{h(I_s)} dB_s
    \\ & = \int_0^t G_s \int_s^t \varphi'(u-s)h'(I_u) duds + \int_0^t \sqrt{h(I_s)} dB_s,
\end{align*}
to obtain their (24).
\end{remark}

\noindent
We now turn to the case $p = \tfrac 12$, where we can rescale the intensity by $\sqrt N$ rather than $N$. This is more involved since the limit of $I^N$ turns out to be a stochastic process. For the proof of the following Theorem, see Section~\ref{ss:proofT2}.

\begin{theorem}[Mean-field limit of multi-variate non-linear Hawkes
  processes, $p = \tfrac 12$]
  \label{T2}
  Let Assumption~\ref{ass1} hold with $p=\tfrac12$ and
  $\theta_N = \tfrac{1}{\sqrt{N}}$. Let
  $Z^N = (Z^{N,1}, ..., Z^{N,N})$ be the multivariate, non-linear
  Hawkes process fom Definition~\ref{def:hawkes}, and
  \begin{align}\label{eq:IN2}
    I^N(t) := \frac 1{\sqrt{N}} \sum_{j=1}^N\int_0^{t-} U_j \varphi(t-s) dZ_s^{j,N}.
  \end{align}
  \begin{enumerate}
  \item For all $w\in\mathbb R$, and a Brownian motion $B$, the
    stochastic integral equation
    \begin{align}
      \label{eq:SDE}
      I_t = I_0 + w\int_0^t \varphi(t-s) h(I_s) ds + \int_0^t
      \varphi(t-s)\sqrt{h(I_s)} dB_s
    \end{align}
    has a unique strong solution.
  \item Let $W\sim N(0,1)$, $B$ be an independent Brownian motion and
    $I$ be the solution of \eqref{eq:SDE} with $w$ replaced by $W$.
    Then, for all $n=1,2,...$,
    $(I^N, Z^{N,1},...,Z^{N,n}) \xRightarrow{N\to\infty} (I, \bar
    Z^1,...,\bar Z^n)$, where $\bar Z^1, ..., \bar Z^n$ are
    conditionally independent given $I$ and $\bar Z^i$ is a simple
    point process with intensity at time $t$ given by $ h(I_t)$,
    $i=1,...,n$.
  \end{enumerate}
\end{theorem}

\noindent
Again, we discuss convergence of the average intensity of $Z^N$. The proof is in Section~\ref{ss:proofCors}.

\begin{corollary}\label{cor2}
Let $Z^N = (Z^{N,1},...,Z^{N,N})$ be as in Theorem~\ref{T2}, and $(I, \bar Z^1, \bar Z^2,...)$ be as in Theorem~\ref{T2}.2. Then, \eqref{eq:cor11}, \eqref{eq:cor12}, \eqref{eq:cor13} hold.
\end{corollary}

\noindent 
We stress that the results corresponding to \eqref{eq:cor14} and \eqref{eq:38} do not hold here since $I^N \Rightarrow I$, so no further upscaling is possible. 

{\color{black}
\begin{remark}[Some connections to \cite{erny2020mean}]
\label{rem:T2a}
Our model and the statements of Theorem~\ref{T2}.2 should be compared to the findings in \cite{erny2020mean}.
In their model, the size of interactions is given by a centered probability measure, so at each spike of a neuron, the potential given to other neurons in the system may change. In contrast, in our model, the potential given to other neurons is fixed for each neuron, but varies from neuron to neuron. The latter case seems more relevant in neuroscience; see the discussion in Section~\ref{S:discuss}. Still, results from Theorem~\ref{T2} can be compare to Theorem 1.4(ii) and Theorem 1.7. of \cite{erny2020mean}, where the same type of convergence is shown, but the interaction kernel $\varphi$ is an exponential function. As we see in Remark~\ref{rem:T2}.2 below, this special choice makes the intensity $I$ a Markov process. However, as the intensity of both models is still comparable, we find (3) from \cite{erny2020mean} in our Remark~\ref{rem:T2}.2 below by inserting $w=0$.
\end{remark}
}

\begin{remark}[Some properties of $I$]\label{rem:T2}
  \begin{enumerate}
  \item If $\varphi'$ exists, it is not hard to see that $I_t$ is a
    semimartingale, since
    \begin{align*}
      \int_0^t \varphi(t-u)\sqrt{h(I_u)} dB_u & =    
                                                \int_0^t \int_0^s \varphi'(t-s)\sqrt{h(I_u)} dB_u ds +
                                                \int_0^t \varphi(0)\sqrt{h(I_u)} dB_u
    \end{align*}
    by the stochastic Fubini Theorem \cite[Theorem 4.A]{berger1980volterra}, where the first term on the
    right hand side has finite variation and the second is a (local)
    martingale. This calculation also shows that $I$ has finite
    variation provided that $\varphi(0)=0$.
  \item Consider the special case $\varphi(s) = e^{-\lambda s}$. Then,
    a solution of \eqref{eq:SDE} also solves
    $$ dI = (wh(I) - \lambda I) dt + \sqrt{h(I)} dB$$
    and in particular is a Markov process.
    In order to see this, we write
    \begin{align*}
      I_t - I_0 = e^{-\lambda t} \Big(w \int_0^t e^{\lambda s} h(I_s) ds + \int_0^t e^{\lambda s} \sqrt{h(I_s)} dB_s\Big).
    \end{align*}
    Using integration by parts (on the processes
    $A_t := e^{-\lambda t}$, which satisfies $dA = -\lambda A dt$, and
    the term $C_t$ in brackets), we find that
    \begin{align*}
      I_t - I_0 & = A_tC_t = \int_0^t A_s dC_s + \int_0^t C_s dA_s
      \\ & = w\int_0^t h(I_s) ds + \int_0^t \sqrt{h(I_s)} dB_s  - \lambda \int_0^t A_s C_s ds
      \\ & = \int_0^t (wh(I_s) - \lambda I_s) ds + \int_0^t \sqrt{h(I_s)} dB_s. 
    \end{align*}
  \end{enumerate}
\end{remark}

\section{Implications for neuroscience}
\label{S:discuss}
 The human brain is very large. It comprises 86 billion neurons, along with other types of cells \cite{Herculano-2009-FrontHumNeurosci}. Each individual neuron makes contact with hundreds or thousands of other neurons, forming a very complex network. It should be thought of as a highly integrated circuit, made of biological parts, which are capable of processing electro-chemical signals and transmitting them to distant neurons with high fidelity. Neurons use electrical impulses (``action potentials'' or ``spikes'') to communicate with other neurons in the same network. Sequences of these impulses evolve in time (``spike trains''), involving the collective activity of many interacting neurons. These multivariate signals are believed to reflect the dynamic ``computations'' performed by these networks. The biological ``purpose'' of these processes covers a wide range from immediate control of bodily functions, like activating a muscle, to abstract information processing, like making a complex decision \cite{Kandel-2012-Book}.

Neurons communicate with other neurons in a network by activating specialized functional contacts between cells (``synapses''). In the most common type of synapses, diffusing chemicals are used for signaling (``neurotransmitters''). In a nutshell, each spike of the sender neuron triggers a transient change in excitability of the receiver neuron. This can either increase or decrease the tendency of the receiver neuron to fire an action potential. We speak of ``excitation'' and ``inhibition'', respectively. A hallmark of normal brain function is a robust balance of excitation and inhibition. Brain diseases, in contrast, are often linked with a marked imbalance of these two forces \cite{Yizhar-2011-Nature}. 

A whole zoo of mathematical models exist to mimic the dynamical processes of synaptic interaction in networks of spiking neurons. Most of them are difficult to analyze, due to the discrete nature of pulse trains and an essential nonlinearity in the spike generation mechanism. From a mathematical point of view, the Hawkes process seems like a very reasonable choice as it can directly reflect several important biological features. First, as a stochastic point process it is by construction based on pulse-coded signaling. Second, the use of interaction kernels automatically reflects the causal nature of communication in a physical system with recurrent loops, allowing only spikes in the past to influence the future evolution of the network. Third, as a multivariate process with coupled components it is able to capture and consistently describe the collective activity dynamics.

The linear Hawkes process \cite{Hawkes-1971-Biometrika} has been employed as a mathematical model of networks of spiking neurons, as it can directly reflect excitatory synaptic interactions between spiking neurons using positive interaction kernels \cite{Jovanovic-2015-PhysRevE}. An important shortcoming of the linear model is that inhibitory synaptic interaction cannot be included in a consistent way. Negative interaction kernels are not allowed, as they might yield a negative point process intensity. To enable the study of networks with excitation-inhibition balance, however, negative interaction kernels must be admitted \cite{Pernice-2011-PLOSCB}. A nonlinear link function can fix this deficit, and effective descriptions of nonlinear biological neurons combining linear synaptic interaction with a nonlinear link function are actually well-known in computational neuroscience. They have been called cascade models \cite{Herz-2006-Science}, and from the viewpoint of statistics they lead to generalized linear models \cite{Paninski-2004-NeuralComput}. For that reason, the nonlinear Hawkes process is a natural extension of Hawkes' classical model, greatly extending its range of applications.

The specific model described in this paper can be further generalized, allowing for a different interaction kernel for each pair of nodes in the network. This includes networks with arbitrary graph topology, setting kernels to zero where a synapse is absent \cite{Pernice-2011-PLOSCB}. In brain networks, the amplitude of individual interaction kernels (``synaptic weight'') is thought to change as a result of learning. The promise for the biologist is to be able to investigate networks, the structure of which is either designed for a specific purpose, or reflects previous experience and memory. The challenge for the mathematician is to manage the lack of symmetry in these networks, and a genuinely multivariate description becomes necessary.

There is an ongoing debate about the significance of scaling synaptic weights for networks of infinite size \cite{Vreeswijk-1998-NeuralComput, Barral-2016-NatureNeurosci}. As the number of neurons in brain networks is typically quite high, such limit may represent a meaningful approximation to large brains. The difficulty is that large brains also have to accommodate constraints with regard to the number of possible connections, imposed by the distance of neurons in space and by the volume occupied by fibers. So it might be more fruitful to rephrase the question a bit \cite{pmid34464597}: Which aspect of the input to a neuron from within the network is more important: Is it the mean drift, or is it the amplitude of the fluctuations? If there is a perfect balance between excitation and inhibition, the drift is zero, and the fluctuations take over and drive the neuron in a stochastic fashion. For the slightest imbalance of the inputs, however, keeping the drift bounded in very large networks forces the fluctuations to be very small, and the dynamics becomes more deterministic. This is where biophysical intuition is in line with the result of mathematical analysis presented in this paper.

\section{Proofs}
We start off in Subsection~\ref{ss:prelim} with some preliminary results on convergence of Poisson processes and convolution equations. We proceed in~\ref{ss:reformulation} with a reformulation of the multivariate linear Hawkes process using a time-change equation. Then, we prove Theorem~\ref{T2} in Sebsection~\ref{ss:proofT2} and Theorem~\ref{T1} in Subsection~\ref{ss:proofT1}. We give the proof of Theorem~\ref{T2} first because it is more involved than the proof of Theorem~\ref{T1}, mainly since the limit process $I$ is stochastic. 

\subsection{Preliminairies}
\label{ss:prelim}\noindent
We need the following (rather standard) convergence results for Poisson processes.

\begin{lemma}[Convergence of Poisson processes]\label{l:poi}
  Let $Y$ be a unit rate Poisson process, defined on some probability
  space $(\Omega, \mathcal F, \mathbb P)$. Then, for all $t>0$,
  \begin{align}\label{eq:lpoi1}
    \sup_{0\leq s\leq t} \frac{1}{N}(Y(Ns)-Ns) \xrightarrow{N\to\infty}0
  \end{align}
  almost surely (and in $L^2$). In addition,
  \begin{align}\label{eq:1b}
    \Big(\frac{1}{\sqrt{N}}(Y(Nt)-Nt)\Big)_{t\geq 0} \xRightarrow{N\to\infty} B
  \end{align}
  for some Brownian motion $B$. We can extend
  $(\Omega, \mathcal F, \mathbb P)$ by a Brownian motion $B$ such that, for all $t\geq 0$,
  \begin{align}\label{eq:1a}
    \sup_{0\leq s\leq t} \Big|\frac{1}{\sqrt{N}}(Y(Ns)-Ns) - B_s\Big| \xrightarrow{N\to\infty} 0
  \end{align}
  almost surely and in $L^2$.
\end{lemma}

\begin{proof}
  For \eqref{eq:lpoi1}, we use Doob's martingale inequality for the
  martingale $(Y(s)-s)_{s\geq 0}$ {\color{black} and the fact that the fourth centered moment of a Poisson random variable with parameter $\lambda$ is $3\lambda^2 + \lambda$} in order to see
  \begin{align*}
    \mathbb E\Big[\sup_{0\leq s\leq t} \Big(\frac{1}{N}(Y(Ns)-Ns)\Big)^4\Big]
    \leq \Big(\frac 43\Big)^4 \mathbb E\Big[ \Big(\frac{1}{N}(Y(Nt)-Nt)\Big)^4 \Big]
    = {\color{black}\Big(\frac 4{3N}\Big)^4}\big(3(Nt)^2+Nt\big),
  \end{align*}
  which is summable. Using Borel-Cantelli, almost-sure convergence follows.
  Next, \eqref{eq:1b} follows by an application of Donsker's theorem
  (\cite{jacod2013limit} Ch VII, Cor. 3.11). By Skorohod's Theorem, we
  can extend our probability space such that this convergence is almost surely with respect to Skorohod distance, and by continuity of $B$ it is equivalent to uniform convergence, i.e. \eqref{eq:1a} holds
  almost surely for all $t\geq 0$. For the $L^2$-convergence in
  \eqref{eq:1a}, it suffices to show the uniform integrability of
  $\sup_{0\leq s\leq t} \Big(\frac{1}{\sqrt{N}}(Y(Ns)-Ns) -
  B_s\Big)^2$. Using Doob's and Minkovski's inequalities, we compute
  \begin{align*}
    \mathbb E\Big[\sup_{0\leq s\leq t} \Big(\frac{1}{\sqrt{N}}(Y(Ns)-Ns) -
    B_s\Big)^4\Big]^{1/4} & \leq \frac 43
                                \mathbb E\Big[\Big(\frac{1}{\sqrt{N}}(Y(Nt)-Nt) -
                                B_t\Big)^4\Big]^{1/4}
    \\ & \leq \frac 43 \Big(\mathbb E\Big[\Big(\frac{1}{\sqrt{N}}(Y(Nt)-Nt)\Big)^4\Big]^{1/4}
         + \mathbb [(B_t)^4]^{1/4}\Big)
    \\ & = \frac 43( (3t^2 + t/N)^{1/4} + (3t^2)^{1/4}) 
  \end{align*}
  showing the desired uniform integrability and $L^2$-convergence
  follows.
\end{proof}

\noindent
In order to bound the value of a convolution equation by its integrator we need the following 
\begin{lemma}\label{l:conv}
  Let $J$ be the sum of an It\^{o} process with bounded coefficients and a c\`{a}dl\`{a}g pure-jump process, and let $\varphi \in \mathcal C_b^1([0,\infty))$. Then
  \begin{align}
    \label{eq:conv}
    \sup_{0\leq s\leq t}\Big(\int_0^s \varphi(s-r)
    dJ_r\Big)^2\leq C(\varphi,t) \cdot \sup_{0\leq s\leq t} J_s^2.   
  \end{align}
\end{lemma}

\begin{proof}
Wlog, we have $J_0=0$. By \cite[Theorem 4.A]{berger1980volterra} and Fubini's theorem for Lebesgue-integrals we can apply the Stochastic Fubini Theorem to $J$, hence
\begin{align*}
    \sup_{0\leq s\leq t}\Big(\int_0^s \varphi(s-r)
    dJ_r\Big)^2 & \notag = \sup_{0\leq s\leq t}\Big(\int_0^s \Big( \varphi(0) + \int_r^s \varphi'(s-u) du \Big) dJ_r\Big)^2 \\
  &  \notag\leq 2 \, ||\varphi||^2 \cdot \sup_{0\leq s\leq t}
    J_{\color{black}s}^2  + 2\sup_{0\leq s\leq
    t}\Big(\int_0^s \int_0^{u} \varphi'(s-u) dJ_r
    du\Big)^2 \\ & \notag= 2 \, ||\varphi||^2 \cdot \sup_{0\leq s\leq t} J_s^2 + 2\sup_{0\leq s\leq t}\Big(\int_0^s
                         \varphi'(s-u) J_u du\Big)^2
  \\ & \leq 2\,(||\varphi||^2
       + t||\varphi'||^2) \cdot \sup_{0\leq s\leq t} J_s^2.
\end{align*}
\end{proof}
\noindent

\subsection{Reformulation of Hawkes processes}
\label{ss:reformulation}
Alternative descriptions of non-linear Hawkes processes have been given in the literature. Above all, the construction using a Poisson random measures is widely used; see e.g.\ Proposition~3 in \cite{delattre2016hawkes}. Here, we rely on the following construction using time-change equations (see e.g.\ Chapter~6 of \cite{EthierKurtz86}), which we give here without proof. 

\begin{lemma}\label{l:reform}
Let $\mathbb G = (\mathbb V, \mathbb E)$, $(\varphi_{ji})_{ji\in \mathbb E}$ and $(h_i)_{i\in\mathbb V}$ be as in Definition~\ref{def:hawkes}, and let Assumption \ref{ass:basic} hold. A point process $Z = (Z^i)_{i\in\mathbb V}$ is a multivariate, non-linear Hawkes process with interaction kernels $(\varphi_{ji})_{ji\in \mathbb E}$ and transfer functions $(h_i)_{i\in\mathbb V}$, if and only if it is the weak solution of the time-change equations
  \begin{align}\label{eq:Z_ref}
  Z_t^i = Y_i\Big( \int_0^t \lambda_s^i ds\Big) = Y_i\Big( \int_0^t h_i\Big(\sum_{j: ji \in \mathbb E} \int_0^{s-}
  \varphi_{ji}(s-u)dZ_u^j\Big) ds \Big),
\end{align}
where $(Y_i)_{i\in \mathbb V}$ is a family of independent unit rate Poisson processes. 
\end{lemma}

\noindent
Under Assumption~\ref{ass1}, the time-change equations from this lemma read, with independent unit rate Poisson processes $Y_1,...,Y_N$,
\begin{align}
  Z_t^{N,i} & = Y_i\Big(\int_0^t h\Big(\sum_{j = 1}^N\int_0^{s-} \theta_N U_j\varphi(s-u)dZ_u^{N,j}\Big)ds\Big).
  \intertext{We rewrite this as}
  \label{eq:timechange0}    Z_t^{N,i} & = Y_i\Big(\int_0^t  h(I_s^N) ds\Big)
  \intertext{with}
  \label{eq:JN} I_s^N & = \int_0^{s-} \varphi(s-u) dJ_u^N \quad \text{ and } \quad J_u^N
                                                            = \theta_N \sum_{j=1}^N U_j Z_u^{N,j}.
\end{align}
Now, we rewrite $J^N$ further. {\color{black} Observe that 
$\sum_{j: U_j=1} Y_j$ is a Poisson process
with rate $\tfrac 12 \sum_{j=1}^N(U_j+1) = pN + \tfrac 12 \sqrt{N}W_N$
(recall $W_N$ from \eqref{eq:CLT}) and
$\sum_{j: U_j=-1} Y_j$ is a Poisson process with
rate $(1-p)N - \tfrac 12 \sqrt{N}W_N$. Hence} we have that, for two independent unit rate Poisson processes $Y^+$ and $Y^-$,
\begin{equation}
  \label{eq:9231}
  \begin{aligned}
    J_t^N & = \theta_N\sum_{j=1}^N U_j Y_j\Big(\int_0^t
    h(I_s^N)ds\Big) \\ & = \theta_N Y^+\Big(\big(pN + \tfrac 12
    \sqrt{N}W_N\big) \int_0^t h(I_s^N)ds \Big) - \theta_N
    Y^-\Big(\big((1-p)N - \tfrac 12 \sqrt{N}W_N\big) \int_0^t
    h(I_s^N)ds \Big).
  \end{aligned}
\end{equation}
We extend our probability space such that the rescaled Poisson processes $Y^+, Y^-$ converge to Brownian motions $B^+, B^-$ almost surely (as in \eqref{eq:1a}).  Consider the Brownian motion $\hat B = (\hat B(t))_{t\geq 0}$ with $\hat B_t = B^+_{pt} - B^-_{(1-p)t}$. Since $\Big(\hat B\Big( \int_0^t h(I_s^N)ds\Big)\Big)_{t\geq 0}$ is a continuous martingale, we can extend the probability space by another Brownian motion $B^N$, such that
\begin{align}
  \label{eq:bm}
  \hat B\Big( \int_0^t h(I_s^N)ds\Big) = \int_0^t \sqrt{h(I_s^N)} dB_s^N \quad \text{for } t\geq 0, 
\end{align}
by \cite[Theorem 16.10]{b2Kallenberg2002}.

\subsection{Proof of Theorem~\ref{T2}}
\label{ss:proofT2}
First, strong existence and uniqueness for \eqref{eq:SDE} follows from \cite[Theorem 3.A]{berger1980volterra}. For the convergence result, we {\color{black}work on a probability space where the convergence $W_N\to W$ from \eqref{eq:CLT} holds almost surely. We} will show in Steps~1~and~2 convergence of $J^N$ to $J$ with
\begin{align}
  \label{eq:J}
  J_t = W\int_0^t h(I_s) ds + \int_0^t \sqrt{h(I_s)} dB_s
\end{align} 
and from $I^N$ to $I$ from \eqref{eq:SDE} with $w$ replaced by $W$. Step~3 then gives the desired convergence of $(I^N,Z^{N,1},...,Z^{N,n})$. We use processes $\widetilde J^N$ and $\widetilde I^N$ which are unique strong solutions of
\begin{align*}
\widetilde{J}^N_t & = W_N \int_{0}^{t} h \big( \widetilde{I}_s^N \big) ds + \int_{0}^{t} \sqrt{h( \widetilde{I}_s^N )} dB_s^N,\\
\widetilde{I}_s^N & = \int_{0}^{s} \varphi(s-r) d\widetilde{J}^N_r,
\end{align*}
where $B^N$ is the Brownian motion defined by \eqref{eq:bm}. {\color{black}Again, strong existence and uniqueness follows from \cite[Theorem 3.A]{berger1980volterra}}. In Step~1, we will show that $J^N - \widetilde J^N \to 0$ and $I^N - \widetilde I^N \to 0$ uniformly on compact sets in $L^2${\color{black}, conditional on $(U_i)_{i\in\mathbb N}$. In Step~2, we show the remaining conditional convergence $\widetilde{I}^N - I \to 0$ and conclude $I^N - I \to 0$ in probability.}
  
\noindent
{{\bf Step 1:} {\em Convergence $J^N - \widetilde J^N \to 0$ and $I^N - \widetilde I^N \to 0$.}}\\
  Recall $Y^+$ and $Y^-$ from  Section~\ref{ss:reformulation}. For $\theta_N = \tfrac{1}{\sqrt{N}}$ and $p = \tfrac 12$ in \eqref{eq:9231} we get that,
  $$J_t^N = \frac{1}{\sqrt{N}} Y^+\Big(\tfrac{N + \sqrt{N}W_N}{2} \int_0^t h(I_s^N)ds \Big) - \frac{1}{\sqrt{N}}
    Y^-\Big(\tfrac{N - \sqrt{N}W_N}{2} \int_0^t
    h(I_s^N)ds \Big).$$
  We write
  \begin{equation}\label{eq:diff}
    \begin{aligned}
      J_t^N - \widetilde{J}_t^N & = A_t^N + W_NC_t^N + D_t^N \quad \text{with}\\
      A_t^N & = \frac{1}{\sqrt{N}} Y^+ \Big( \frac{N + \sqrt{N} W_N
      }{2}\int_0^t h(I_s^N) ds \Big)
      - \frac{1}{\sqrt{N}} Y^- \Big( \frac{N - \sqrt{N} W_N }{2}\int_0^t h(I_s^N) ds \Big)\\
      & \qquad \qquad \qquad \qquad \qquad \qquad \qquad \qquad - W_N
      \int_0^t h(I_s^N) ds - \hat B\Big(\int_0^t h(I_s^N) ds\Big),\\
      C_t^N & =  \int_0^t h(I_s^N) ds - \int_0^t h(\widetilde{I}_s^N) ds,\\
      D_t^N & = \int_0^t \sqrt{h(I_s^N)} dB^N_s - \int_0^t
      \sqrt{h(\widetilde I_s^N)} dB^N_s.
    \end{aligned}
  \end{equation}
We make the disclaimer, that we will use $C>0$ for a variable, which only depends on $t, h, \varphi$, and which might change from line to line. {\color{black} We write $\mathbb E_U$ for the conditional expectation with respect to $U=(U_i)_{i\in\mathbb N}$.} The function $h$ is bounded, therefore, for some unit rate Poisson process $Y$,
  \begin{equation}
    \label{eq:9123}
    \begin{aligned}
      \mathbb E{\color{black}_U}[\sup_{0\leq s\leq t} (A_s^N)^2] & \leq \mathbb
      E{\color{black}_U}\Big[\sup_{0\leq s\leq t||h||} \Big(\frac{1}{\sqrt{N}} Y^+
      \Big( \frac{N + \sqrt{N} W_N }{2}s \Big)
      - \frac{1}{\sqrt{N}} Y^- \Big( \frac{N - \sqrt{N} W_N }{2} s \Big)\\
      & \qquad \qquad \qquad \qquad \qquad \qquad \qquad \qquad \qquad
      \qquad \qquad - W_N s - \hat B(s)\Big)^2\Big] \\ & \leq C \cdot
      \Big(\mathbb E{\color{black}_U}\Big[\sup_{0\leq s\leq t||h||} \Big(
      \frac{1}{\sqrt{N}}(Y^+(Ns/2) - Ns/2) - B^+(s/2)\Big)^2\Big] \\ &
      \qquad + \mathbb E{\color{black}_U}\Big[\sup_{0\leq s\leq t||h||} \Big(
      \frac{1}{\sqrt{N}}(Y^-(Ns/2) - Ns/2) - B^-(s/2)\Big)^2\Big] \\ &
      \qquad + \mathbb E{\color{black}_U}\Big[\sup_{0\leq s\leq |W_N|t||h||} \Big(
      \frac{1}{\sqrt{N}}(Y(\sqrt N s)) - s\Big)^2\Big]\Big)
      \\ & {\color{black} \leq o(1) + C \cdot \mathbb E_U\Big[\Big(
      \frac{1}{\sqrt{N}}(Y(\sqrt N |W_N|t||h||s)) - |W_N|t||h||s\Big)^2\Big]}
      \\ & {\color{black}\leq o(1) + C \frac 1{\sqrt{N}}W_N}
      \xrightarrow{N\to\infty} 0
    \end{aligned}
  \end{equation}
  {\color{black}almost surely,} by Lemma~\ref{l:poi} and Doob's inequality. Then, using Jensen,
  \begin{align*}
    \mathbb E{\color{black}_U}[\sup_{0\leq s\leq t} (W_NC_s^N)^2]
    & {\color{black}\leq \sup_N W_N^2} \mathbb E{\color{black}_U}\Big[\sup_{0\leq s\leq t} \Big( \int_0^s h(I_r^N) - h(\widetilde I_r^N) dr\Big)^2\Big] 
    \\ & \leq {\color{black}\sup_N W_N^2}
      t h_{Lip}^2 \int_0^t \mathbb E{\color{black}_U}[ (I_s^N - \widetilde I_s^N)^2]  ds
  \end{align*}
  {\color{black}Note that $\sup_N W_N^2<\infty$, as $W_N$ converges to $W$, which is finite almost surely.} By Doob's inequality and Ito's isometry,
  \begin{align*}
    \mathbb E{\color{black}_U}[\sup_{0\leq s\leq t} (D_s^N)^2]
    & \leq C\cdot 
      \mathbb E{\color{black}_U}[(D_t^N)^2] = C\cdot \mathbb E{\color{black}_U} \Big[\Big( \int_0^t \sqrt{h(I_s^N)} dB^N_s - \int_0^t
      \sqrt{h(\widetilde{I}_s^N)} dB^N_s \Big)^2 \Big]
    \\& = C\cdot \bE{\color{black}_U} \Big[ \int_0^t \big(\sqrt{h(I_s^N)} - \sqrt{h(\widetilde{I}_s^N)}\big)^2 ds \Big]\\
    & \leq C\cdot \int_0^t \mathbb E{\color{black}_U}[(I_s^N - \widetilde{I}_s^N)^2 ] ds.
  \end{align*}
  In order to bound $\sup_{0\leq s\leq t} (W_N C_s^N)^2$ and
  $\sup_{0\leq s\leq t} (D_s^N)^2$ by $\sup_{0\leq s\leq t}(J_s^N - \widetilde J_s^N)^2$,
  apply Lemma~\ref{l:conv} to $J_s^N - \widetilde{J}_s^N$. We obtain
  \begin{align*}
    \mathbb E{\color{black}_U} \big[ \sup_{s \leq t} ( J_s^N - \widetilde{J}_s^N )^2 \big]
    & \leq C\cdot 
      \mathbb E{\color{black}_U}[\sup_{0\leq s\leq t} (A_s^N)^2] +
      C\cdot \int_0^t \mathbb E{\color{black}_U}[(I_s^N - \widetilde{I}_s^N)^2 ] ds
    \\ & \leq C\cdot \mathbb E{\color{black}_U}[\sup_{0\leq s\leq t} (A_s^N)^2]
         + C\cdot \int_0^t \mathbb E{\color{black}_U}[\sup_{r\leq s} (J_r^N - \widetilde{J}_r^N)^2 ] ds.
  \end{align*}
  By the Gronwall inequality and \eqref{eq:9123}, we conclude that,
  for all $t\geq 0$,
  \begin{align*}
    \mathbb E{\color{black}_U} \big[ \sup_{s \leq t} ( J_s^N - \widetilde{J}_s^N )^2 \big]
    & \xrightarrow{N\to\infty} 0
      \intertext{and with Lemma~\ref{l:conv} also}
      \mathbb E{\color{black}_U} \big[ \sup_{s \leq t} ( I_s^N - \widetilde{I}_s^N )^2 \big] & \xrightarrow{N\to\infty} 0,
  \end{align*}
  {\color{black}almost surely.}
  
  \noindent
  {{\bf Step 2:} {\em Convergence $\widetilde I^N - I \to 0$}}\\
  \noindent
  Since $B^N$ is a Brownian motion, we have that, in distribution,
  \begin{align*}
    \widetilde I_t^N & = W_N \int_0^t \varphi(t-s) h(\widetilde I_s^N) ds + \int_0^t
                       \varphi(t-s) \sqrt{h(\widetilde I_s^N)} dB_s^N,
    \\
    I_t & = W \int_0^t \varphi(t-s) h(I_s) ds + \int_0^t \varphi(t-s) \sqrt{h(\widetilde I_s)} dB_s^N.
  \end{align*}
  Therefore, on this probability space,
  \begin{align*}
    \widetilde I_t^N - I_t & = E_t^N + F_t^N + G_t^n \text{ with}\\
                             E_t^n & := (W_N-W)\int_0^t \varphi(t-s) h(I_s) ds,\\
    F_t^N & := W_N \int_0^t \varphi(t-s) (h(\widetilde I_s^N) - h(I_s)) ds,\\
    G_t^N & := \int_0^t \varphi(t-s) \Big(\sqrt{h(\widetilde I_s^N)} - \sqrt{h(I_s)}\Big) dB_s^N.
  \end{align*}
  Since, again using some constant $C<\infty$, which can change from
  line to line,
  \begin{align*}
    \mathbb E{\color{black}_U}[\sup_{0\leq s\leq t} (E_s^N)^2] & \leq C \cdot {\color{black}(W_N-W)^2} \xrightarrow{N\to\infty} 0,\\
    \mathbb E{\color{black}_U}[\sup_{0\leq s\leq t} (F_s^N)^2] & \leq C \cdot {\color{black}\sup_N W_N^2}\int_0^t \mathbb E{\color{black}_U}[\sup_{0 \leq r\leq s} (\widetilde I_r^N - I_r)^2 ] ds,\\
    \mathbb E{\color{black}_U}[\sup_{0\leq s\leq t} (G_s^N)^2] & \leq C \cdot \int_0^t \mathbb E{\color{black}_U}[\sup_{0 \leq r\leq s} (\widetilde I_r^N - I_r)^2 ] ds,
    \intertext{  we find that}
    \mathbb E{\color{black}_U}[\sup_{0 \leq s\leq t} (\widetilde I_s^N - I_s)^2 ] & \leq C \cdot {\color{black}(W_N-W)^2} + C \cdot \int_0^t \mathbb E{\color{black}_U}[\sup_{0 \leq r\leq s} (\widetilde I_r^N - I_r)^2 ] ds,
  \end{align*}
  and the Gronwall inequality again implies that
  $\mathbb E{\color{black}_U}[\sup_{0 \leq s\leq t} (\widetilde I_s^N - I_s)^2 ]
  \xrightarrow{N\to\infty} 0${\color{black}, almost surely. Combining Step~1 and Step~2 we obtain $\mathbb E_U[\sup_{0 \leq s\leq t} (I_s^N - I_s)^2 ]
  \xrightarrow{N\to\infty} 0$. We deduce $\mathbb E_U[ 1 \wedge \sup_{0 \leq s\leq t} (I_s^N - I_s)^2 ]
  \xrightarrow{N\to\infty} 0$ and, by dominated convergence, $\mathbb E[1 \wedge\sup_{0 \leq s\leq t} (I_s^N - I_s)^2 ]
  \xrightarrow{N\to\infty} 0$. Hence we have shown that $\sup_{0 \leq s\leq t} (I_s^N - I_s)^2\to 0$ in probability.} 

  \noindent
{\bf Step 3:} {\em Convergence of $(I^N,Z^{N,1},...,Z^{N,n})$}\\
\noindent
\sloppy The predictable quadratic variation of the compensated process $Z^{N,i}_t-\int_0^t h(I_s^N) ds$ is given by $\langle Z^{N,i} - \int_0^\cdot h(I_s^N) ds \rangle_t = \int_0^t h(I_s^N) ds$. By \cite{jacod2013limit}[Chapter VI, Corollary 3.33, Theorem 4.13], tightness of $(Z^{N,1},...,Z^{N,n})$ follows from convergence of $\sum_{i\leq n} \langle Z^{N,i} - \int_0^\cdot h(I_s^N) ds \rangle_t = n \int_0^t h(I_s^N) ds$. We already know that $I^N \Rightarrow I$, hence $(I^N,Z^{N,1},...,Z^{N,n})$ is tight. To identify the limit, let us try to extend the state space such that the Hawkes process becomes Markovian. Recall that each $Z^{N,i}$ jumps at time $t$ at rate $h(I^N(t))$ with $I^N(t) = \int_0^t \varphi(t-s) dJ^N_s$ and $J_s^N$ as given in \eqref{eq:JN}. The process $J^N$ jumps at time $t$ from $x$ to $x + \theta_N$ at rate $h\Big( \int_0^t \varphi(t-s) dJ^N_s\Big) \cdot u^+$ with $u^+ := \sum_{j} 1_{U_j=1}$ and to $x - \theta_N$ at rate $h\Big( \int_0^t \varphi(t-s) dJ^N_s\Big) \cdot u^-$ with $u^- = \sum_{j}1_{U_j=-1} = N - u^+$.  Note that the process $J^N$ satisfies, for smooth $f$,
\begin{align*}
& \frac{\mathbb E\Big[ f(J^N_{t+\varepsilon})
	- f(J^N_t) | \mathcal F_t \Big]}{\varepsilon}
\\ & = h(I_t^N) u^+ \Big[ f(J^N_t + \theta_N )  - f(J^N_t) \Big]
+ h(I_t^N) u^- \Big[ f(J^N_t - \theta_N )  - f(J^N_t) \Big] + o(\varepsilon)
\\ & =  f'(J^N_t) \cdot 
\Big( h(I_t^N) \underbrace{(u^+_N - u^-_N) \theta_N}_{=W_N} \Big) + \tfrac 12 f''(J^N_t) \cdot h(I_t^N) \underbrace{(u^+_N + u^-_N)\theta_N^2}_{=1}   + o(\varepsilon)
\\ & =  f'(J^N_t) \cdot h(I_t^N) W_N + \tfrac 12 f''(J^N_t) \cdot h(I_t^N) + o(\varepsilon).
\end{align*}
We obtain a similar expression if we ignore the jumps of $Z^{N,i}, i=1,...,n$, i.e. if we replace $J^N$ by $J^{N(n)}=\theta_N \sum_{j>n}U_jZ^{N,j}$. In this case we have to replace $W_N$ by $W_{N(n)}=\theta_N \sum_{j>n}U_j$ and $(u^+_N + u^-_N)\theta_N^2$ by $(N-n)/N$. In the limit $N\to\infty$ we obtain the same expression, as $W_{N(n)}\to W$ and  $(N-n)/N \to 1$. From now on write $J_{[0,t]}$ for the path of a process $J$ on the interval $[0,t]$ and
\begin{align*}
  G_J^{N(n)}f(J_{[0,t]}^{\color{black}N}) & = f'(J_t^{\color{black}N}) \cdot 
  h(I_t^{\color{black}N}) W_{N(n)} + \tfrac {N-n}{2N} f''(J_t^{\color{black}N}) \cdot h(I_t^{\color{black}N})
  \\ G_Jf(J_{[0,t]}) & = f'(J_t) \cdot 
  h(I_t) W + \tfrac 12 f''(J_t) \cdot h(I_t)
\end{align*}
The generator $G^N$ of the Markov process $\big(J^N_{[0,t]},Z^{N,1}_t,...,Z^{N,n}_t\big)$ with domain $\mathcal D$ consisting of functions of the form $f_{0,...,n}(J_{[0,t]},k_1,...,k_n)=f_0(J_t)f_1(k_1)...f_n(k_n)$, where we suppose the functions $f_1,...,f_n$ to be bounded functions on $\mathbb N$, and $f_0$ to be a smooth function on $\mathbb R$ with bounded second derivative, is given by
\begin{align*}
  G^Nf_{0,..,n}(J_{[0,t]}^{\color{black}N},Z^{N,1}_t&,...,Z^{N,n}_t) =  G^{N(n)}_Jf_0(J_{[0,t]}^{\color{black}N})\prod_{i=1}^n f_i(Z^{N,i}_t) 
  \\ &+ h(I_t)\sum_{j = 1}^n\big( f_0(J_t + \theta_N U_j)f_j(Z^{N,j}_t+1) - f_0(J_t)f_j(Z^{N,j}_t) \big) \prod_{i \neq j} f_i(Z^{N,i}_t).
\end{align*}
By \cite{EthierKurtz86}[Chapter 4, Proposition 1.7],
\begin{align*}
f_0(J^N_t)\prod_{i=1}^{n}f_i(Z^{N,i}_t) - \int_0^t G^N f_0...f_n(J_{[0,s]}^{\color{black}N},Z^{N,1}_s,...,Z^{N,n}_s) ds
\end{align*}
is a martingale with respect to the filtration generated by $J,Z^{N,i}$. It is equivalent that
\begin{align}
\mathbb E\Big[ \Big(f_0(J^N_t)\prod_{i=1}^{n}f_i(Z^{N,i}_t) - f_0(J^N_s)\prod_{i=1}^{n}f_i(Z^{N,i}_s) - \int_s^t G^N f_0...f_n(J_{[0,r]}^{\color{black}N},Z^{N,1}_r,...,Z^{N,n}_r) ds\Big)\nonumber
\\ \prod_{j=1}^{m}g_j(J_{[0,t_j]}^{\color{black}N},Z^{N,1}_{t_j},...,Z^{N,n}_{t_j}) \Big]=0, \label{eq:MP}
\end{align}
for all $m\in\mathbb N, 0\leq t_1 \leq ... \leq t_m \leq s < t$ and continuous functions $g_j$. To identify the limit, assume $\big(J^N,Z^{N,1},...,Z^{N,n}\big) \Rightarrow \big(\bar J,\bar Z^{1},...,\bar Z^{n}\big)$ in distribution with respect to Skorohod distance. By continuity of projection, $J^N \Rightarrow \bar J$, hence $\bar J=J$. In Step 1 and 2 we have already shown the stronger convergence $J^N \to J$ and $I^N \to I$ with respect to uniform topology. As $G^Nf$ converges to some $Gf$, $\theta_N\to 0$ and the evaluation at timepoints $t$ is continuous with respect to uniform topology, we can pass to the limit in \eqref{eq:MP} and obtain that 
\begin{align}\label{eq:MPlimit}
f_0(J_t)\prod_{i=1}^{n}f_i(\bar Z^{i}_t) - \int_0^t G f_{0,...,n}(J_s,\bar Z^{1}_s,...,\bar Z^{n}_s) ds
\end{align}
is a martingale for any $f_{0,...,n} \in \mathcal D$, where 
\begin{align*}
Gf_{0,..,n}(J_{[0,t]},Z^{1}_t,...,Z^{n}_t) = & G_Jf_0(J_{[0,t]})\prod_{i=1}^n f_i(Z^{i}_t) 
\\ & + h(I_t)f_0(J_t)\sum_{j = 1}^n\big( f_j(Z^{j}_t+1) - f_j(Z^{j}_t) \big) \prod_{i \neq j} f_i(Z^{i}_t).  
\end{align*}
It remains to read off the distribution of $(J,\bar Z^1,...,\bar Z^n)$ from \eqref{eq:MPlimit}. Choose $f_0 = 1$ and $f_i(k) = \exp(-u_i k)$ for $u_j\geq 0$. Then 
\begin{align*}
0 = & \mathbb E \Big[ f_0(J_t)\prod_{i=1}^{n}f_i(\bar Z^{i}_t) - \int_0^t G f_0...f_n(J_s,\bar Z^{1}_s,...,\bar Z^{n}_s) ds \Big]
\\ = & \mathbb E \big[ \prod_{i=1}^{n} e^{-u_i \bar Z^{i}_t} - \int_0^t h(I_s)\sum_{j = 1}^n \big( e^{-u_j}-1 \big) \prod_{i=1}^n e^{-u_i \bar Z^{i}_s} ds \big]
\\ = & \mathbb E \Big[ \mathbb E \big[\prod_{i=1}^{n} e^{-u_i \bar Z^{i}_t} | I \big] - \int_0^t h(I_s)\sum_{j = 1}^n \big( e^{-u_j}-1 \big) \mathbb E \big[ {\color{black}\prod_{i=1}^n } e^{-u_i \bar Z^{i}_s} | I \big] ds \Big],
\end{align*}
{\color{black}where, in the last equality, we have interchanged the $ds$-integral and conditional expectation with respect to $I$, and used the $I$-measurability of $h(I)$}. The function in brackets is nonnegative, hence
\begin{align*}
\mathbb E \Big[\prod_{i=1}^{n} e^{-u_i \bar Z^{i}_t} | I \Big] = \prod_{j = 1}^n
\exp \Big( (e^{-u_j}-1) \int_0^t h(I_s) ds \Big).
\end{align*}
As claimed, given $I$ the processes $\bar Z^i$ are independent poisson processes with intensity at time $t$ given by $h(I_t)$.

\subsection{Proof of Theorem~\ref{T1}}
\label{ss:proofT1}
{\bf Proof of 1.:}\\
Define
\begin{align}
  J_t = (2p-1) \int_0^t h(I_s)ds, \text{ such that } I_t = \int_0^t \varphi(t-s)dJ_s.
\end{align}
For $\theta_N = \tfrac{1}{N}$ in \eqref{eq:9231} we get that,
\begin{align*}
  J_t^N - J_t & = \frac 1N\Big( Y^+\Big(\big( pN + \tfrac 12 \sqrt{N}W_N\big) \int_0^t h(I_s^N) ds\Big) -
                Y^+\Big(pN \int_0^t h(I_s^N) ds\Big)\Big)
  \\  & \quad - \frac 1N\Big( Y^-\Big(\big( (1-p)N - \tfrac 12 \sqrt{N}W_N\big) \int_0^t h(I_s^N) ds\Big) -
        Y^-\Big((1-p)N \int_0^t h(I_s^N) ds\Big)\Big)
  \\ & \qquad \qquad + \frac 1N Y^+\Big(pN \int_0^t h(I_s^N) ds\Big) - p \int_0^t h(I_s^N) ds
  \\ & \qquad \qquad - \frac 1N Y^-\Big((1-p) N \int_0^t h(I_s^N) ds\Big) + ( 1 - p ) \int_0^t h(I_s^N) ds
  \\ & \qquad \qquad \qquad \qquad \qquad \qquad \qquad + (2p-1) \int_0^t h(I_s^N) - h(I_s) ds,
\end{align*}
so
\begin{align}
  \notag
  \sup_{0\leq s\leq t} ( J_s^N - J_s )^2
  &\leq  \varepsilon_N + C \cdot \Big(\sup_{0\leq s\leq t||h||} \Big(\frac 1N
       Y^+\Big(pN s \Big)
       - p s\Big)^2 + \Big(\frac 1N Y^-\Big((1-p)Ns \Big) - (1-p) s\Big)^2\Big) \notag
  \\ & \notag \qquad \qquad \qquad + (2p-1)^2 h_{Lip} \cdot \sup_{0\leq s\leq t} \int_0^s (I_r^N - I_r)^2)ds
  \\ &   \label{eq:721}\qquad \qquad \qquad \qquad \qquad \qquad \leq \varepsilon_N + C \cdot
       \int_0^t \sup_{0\leq r\leq s}(I_r^N - I_r)^2ds
\end{align}
with $\varepsilon_N \xrightarrow{N\to\infty}0$ (by Lemma~\ref{l:poi}). 
By the Gronwall inequality and Lemma~\ref{l:conv} we can conclude
\begin{align}\label{eq:IN0}
  \sup_{0\leq s\leq t} \Big( J_s^N - J_s \Big)^2 \xrightarrow{N\to\infty} 0\qquad \text{and}\qquad \sup_{0\leq s\leq t} \Big( I_s^N - I_s \Big)^2 \xrightarrow{N\to\infty} 0.
\end{align}

\noindent
{\bf Proof of 2.:}\\
Define $\bar Z^i_t = Y_i\big( \int_0^t h(I_s)ds\big)$, where $Y_1,...,Y_N$ are independent Poisson process as in \eqref{eq:timechange0}. Fix $\omega\in\Omega$, such that $\sup_{0\leq s\leq t} | I_s^N(\omega) - I_s | \rightarrow 0$. As $h$ is Lipschitz, $\int_0^t h(I_s^N(\omega))ds \rightarrow \int_0^t h(I_s)ds$, hence for any point of continuity of $t\mapsto \bar Z^i_t(\omega)$ we can conclude $Z^i_t(\omega) \rightarrow \bar Z^i_t(\omega)$. As the points of continuity are dense in $[0,\infty)$, convergence in Skorohod-distance follows from \cite{jacod2013limit}[Theorem 2.15 c)(ii)]. 

\noindent
{\bf Proof of 3.:}\\
First, strong existence and uniqueness for \eqref{eq:SDE_K}
follows from \cite{berger1980volterra}. The proof proceeds similarly to the proof of Theorem \ref{T2}, Step~1~and~2. Wlog, we assume that $W_N \to W$ almost surely and in $L^2$. Recall the pair $(G,K)$ from \eqref{eq:SDE_K} (for some Brownian motion $B$). 
Further, we define the auxiliary processes $\widetilde{G}^N,\widetilde{K}^N$ by (recall from \eqref{eq:bm} the Definition of the Brownian motion $B^N$)
\begin{equation}\label{eq:tildeK}
\begin{aligned}
  \widetilde{K}^N_t & = \int_0^t \varphi(t-s)d\widetilde{G}^N_s
  \\ \widetilde{G}^N_t & = \int_0^t Wh(I_s) + (2p-1) h'(I_s) \widetilde{K}^N(s)ds
  + \int_0^t \sqrt{h(I_s)} dB^N_s.
\end{aligned}
\end{equation}
In our proof, we will first show convergence of $\sqrt N (J^N-J) - \widetilde{G}^N \rightarrow 0$ and $\sqrt N (I^N-I) - \widetilde{K}^N \rightarrow 0$. We then replace $B^N$ by $B$ and find $\widetilde{K}^N \Rightarrow K$, since $(\widetilde K^N, \widetilde G^N) \sim (K,G)$. The process $\sqrt N (J^N-J) - \widetilde{G}^N$ is described by the following equation:
\begin{align*}
  & \sqrt N (J^N_t - J_t) - \widetilde{G}^N_t = A^N_t + C^N_t + D^N_t + (2p-1) E^N_t 
  \intertext{with}
  A_t^N & = \frac{1}{\sqrt{N}} Y^+ \Big( \Big(pN + \frac{\sqrt{N} W_N
  }{2}\Big)\int_0^t h(I_s^N) ds \Big)
  - \frac{1}{\sqrt{N}} Y^- \Big( \Big((1-p)N - \frac{\sqrt{N} W_N }{2}\Big)\int_0^t h(I_s^N) ds \Big)\\
  & \qquad \qquad \qquad \qquad - \sqrt N (2p-1) \int_0^t h(I_s^N) ds - W_N
  \int_0^t h(I_s^N) ds - \hat B\Big(\int_0^t h(I_s^N) ds\Big),\\
  C_t^N & =  W_N \int_0^t h(I_s^N) ds - W \int_0^t h({I}_s) ds,\\
  D_t^N & = \int_0^t \sqrt{h(I_s^N)} dB^
  N_s - \int_0^t
  \sqrt{h(I_s)} dB^N_s,\\
  E_t^N & = \int_0^t \sqrt N \big(h(I^N_s)-h(I_s)\big) - h'(I_s)\widetilde{K}^N_s ds.
\end{align*}
As in the proof of Theorem~\ref{T2}, Step~1, we can show that $C^N, D^N \to 0$ (as we already know that $I^N\to I$ and $W_N \rightarrow W$) uniformly on compact time intervals in $L^2$. The uniform $L^2$-convergence $A^N \to 0$ follows from Lemma~\ref{l:poi}. For $E^N$ we obtain
\begin{align*}
  E^N_t = \int_0^t \big(\sqrt N(I^N_s-I_s)-\widetilde{K}^N_s\big)h'(I^N_s)+\sqrt N R_1h(I^N_s,I_s) ds,
\end{align*}
where $R_1h(I^N_s,I_s)$ is the first order remainder in Taylor's formula. Using the Peano form of the remainder we obtain for some $\xi^N_s$ between $I^N_s$ and $I_s$ 
\begin{align*}
  \Big(\sqrt N R_1h(I^N_s,I_s)\Big)^2& =   N \Big((h'(\xi^N_s) - h'(I_s))^2 (I^N_s-I_s)^2\Big)\\
  & \leq (h'_{Lip})^2 N (I^N_s-I_s)^4.
\end{align*}
Therefore
\begin{align*}
  \mathbb E\Big[ \sup_{0\leq s\leq t} \big(E^N_s\big)^2 \Big] & \leq 2 \mathbb E\Big[ \sup_{0\leq s\leq t} \Big(\int_0^s \big(\sqrt N(I^N_u-I_u)-\widetilde{K}^N_u\big)h'(I^N_u) du\Big)^2 \Big]\\
  & \qquad \qquad + 2\mathbb  E\Big[ \sup_{0\leq s\leq t} \Big(\int_0^t \big(\sqrt N R_1h(I^N_u,I_u)\big)du \Big)^2 \Big]\\
  & \leq 2t ||h'||^2 \int_0^t \mathbb E\Big[ \big(\sqrt N(I^N_s-I_s)-\widetilde{K}^N_s\big)^2 \Big]ds\\
  & \qquad \qquad + 2t^2(h'_{Lip})^2 \mathbb E\Big[ \sup_{0\leq s\leq t} N (I^N_s-I_s)^4 \Big].
\end{align*}
We can show convergence of the second term similarly to the proof of 1., as 
\begin{align*}
  N\cdot \mathbb E\Big[\sup_{0\leq s\leq t} \Big(\frac{1}{N}(Y^\pm(Ns)-Ns)\Big)^4\Big]
  & \leq \Big(\frac 43\Big)^4 \frac{1}{N^3} \mathbb E\Big[ \Big((Y^\pm(Nt)-Nt)\Big)^4 \Big]\\
  & = \Big(\frac 43\Big)^4 \Big(\frac{3t^2}{N} + \frac{t}{N^2}\Big) \xrightarrow{N\to\infty} 0.
\end{align*}
For the first term we apply Lemma~\ref{l:conv} in order to obtain
\begin{align*}
  \mathbb E\Big[ \sup_{0\leq s\leq t} \Big(\big(\sqrt N(I^N_s-I_s)-\widetilde{K}^N_s\big)\Big)^2 \Big] \leq C \cdot \mathbb E\Big[ \sup_{0\leq s\leq t} \Big(\big(\sqrt N(J^N_s-J_s)-\widetilde{G}^N_s\big)\Big)^2 \Big].
\end{align*}
Summing up the results on $A^N, C^N, D^N$ and $E^N$ we conclude
\begin{align*}
  \mathbb E\Big[ \sup_{0\leq s\leq t} \big(\sqrt N (J^N_s-J_s) - \widetilde{G}^N_s\big)^2 \Big]\leq \varepsilon_N + C \int_0^t \mathbb E\Big[ \sup_{0\leq u\leq s} \big(\sqrt N (J^N_u-J_u) - \widetilde{G}^N_u\big)^2 \Big] ds
\end{align*}
for $\varepsilon_N \to 0$ and by the Gronwall inequality
\begin{align*}
  \mathbb E\Big[ \sup_{0\leq s\leq t} \big(\sqrt N (J^N_s-J_s) - \widetilde{G}^N_s\big)^2 \Big]\xrightarrow{N\to\infty} 0.
\end{align*} 
Applying Lemma~\ref{l:conv} it follows that
\begin{align*}
  \mathbb E\Big[ \sup_{0\leq s\leq t} \big(\sqrt N (I^N_s-I_s) - \widetilde{K}^N_s\big)^2 \Big]\xrightarrow{N\to\infty} 0.
\end{align*}

\subsection{Proof of Corollaries \ref{cor1} and \ref{cor2}}
\label{ss:proofCors}

\begin{proof}[Proof of Corollary \ref{cor1}]
Recall $Y^+$ and $Y^-$ from Section~\ref{ss:reformulation}, and which are used frequently in the proof of Theorem~\ref{T1}. For \eqref{eq:cor11}, using \eqref{eq:timechange0} and Lemma \ref{l:poi}, for the Poisson process $Y = (Y(t))_{t\geq 0}$, given by $Y(t) = Y^+(pt) + Y^-((1-p)t)$,
\begin{equation}\label{eq:cor11b}
\begin{aligned}
    \sup_{0\leq t \leq T} & \Big|\frac{1}{N}\sum_{j=1}^N \Big(Z_t^{N,j} - \int_0^t h(I_s^N) ds\Big)\Big| = \sup_{0\leq t \leq T} \Big| \frac 1N Y\Big(N \int_0^t h(I_s^N) ds\Big) - \int_0^t h(I_s^N) ds\Big|
    \\ & \leq \sup_{0\leq t \leq T||h||} \Big|\frac 1N (Y(Nt) - Nt)\Big| = 0.
\end{aligned}
\end{equation}
For \eqref{eq:cor12}, by the law of large numbers, $\frac 1N \sum_{j=1}^N \bar Z^j_T \xrightarrow{N\to\infty} \int_0^T h(I_t) dt$, as well as $h(I^N) \xrightarrow{N\to\infty} h(I)$ in $L^2$, uniformly on compact time intervals by Theorem~\ref{T1}.1, therefore
\begin{align*}
    \frac 1N \sum_{j=1}^N\Big( Z^{N,j} - \bar Z^j\Big) & = 
    \frac 1N \sum_{j=1}^N\Big( Z^{N,j} - \int_0^. h(I_s^N) ds\Big) \\ & \qquad \qquad - 
    \frac 1N \sum_{j=1}^N\Big( \bar Z^{j} - \int_0^. h(I_s)ds\Big)
     + \int_0^. (h(I_s^N) - h(I_s)) ds \xrightarrow{N\to\infty} 0.
\end{align*}
For \eqref{eq:cor13}, we write, with the same Poisson process as in \eqref{eq:cor11b}, and with the Brownian motions $B'$ arising as limit of the compensated Poisson process $Y$, and some $\widetilde B$,
\begin{equation}\label{eq:s3}
\begin{aligned}
\frac{1}{\sqrt N}& \sum_{j=1}^N
\Big(Z^{N,j}_t - \int_0^t h(I_s^N)ds\Big) 
= \frac{1}{\sqrt N} \Big(Y\Big( N\int_0^t h(I_s^N)ds\Big) - N\int_0^t h(I_s^N)ds\Big) 
\\ & = B'\Big(\int_0^t h(I_s)ds\Big) + o(1) =  \int_0^t \sqrt{h(I_s)}d\widetilde B_s + o(1).
\end{aligned}
\end{equation}
Next, for \eqref{eq:cor14}, with $K$ as in Theorem~\ref{T1}.3,
\begin{align*}
\sqrt N (h(I_t^N) - h(I_t)) \Big)=\sqrt N h'(I_t) (I_t^N - I_t) + o(1) = h'(I_t) K_t +  o(1).
\end{align*}
Last, we note that the left hand side of \eqref{eq:38} is a sum of the left hand sides of \eqref{eq:cor13} and \eqref{eq:cor14}. We obtain the result by summing the limits of these equations, once we determine the correlation structure of the martingales $\int_0^. \sqrt{h(I_s)} dB_s$ in \eqref{eq:SDE_K} and $\int_0^. \sqrt{h(I_s)} d\widetilde B_s$ in \eqref{eq:cor13} on a joint probability space. We will show that joint weak convergence of \eqref{eq:cor13} and \eqref{eq:cor14} holds if 
\begin{align}
    \label{eq:corr1}
    \mathbb E[\int_0^t \sqrt{h(I_s)} dB_s \int_0^t \sqrt{h(I_s)} d\widetilde B_s] = (2p-1)\int_0^t h(I_s)ds,    
\end{align}
implying the result. Recall from Section~\ref{ss:reformulation} the limits $B^+$ and $B^-$ of the rescaled Poisson processes $Y^+$ and $Y^-$. By the construction of $Y$ above \eqref{eq:cor11b}, we find that $B'_t = B^+_{pt} + B^-_{(1-p)t}$ is the limit of the rescaled Poisson process $Y$. In addition, $\hat B_t = B^+_{pt} - B^-_{(1-p)t}$ arises above \eqref{eq:bm}. Clearly, by independence of $B^+$ and $B^-$, we have $\mathbb E[\hat B_t B_t'] = \mathbb E[(B_{pt}^+ + B_{(1-p)t}^-)(B_{pt}^+ - B_{(1-p)t}^-)] = (2p-1)t$.

From the proof of Theorem~\ref{T1}.3, we know that the convergence in this result is even in probability if we exchange $K$ by $\widetilde K^N$ from \eqref{eq:tildeK}. Hence, using $I^N \to I$, we find for the integral appearing in \eqref{eq:SDE_K}, on a joint probability space, and taking \eqref{eq:bm} into account,
$$\int_0^t \sqrt{h(I_s)} dB_s = \lim_{N\to\infty} 
\int_0^t \sqrt{h(I^N_s)} dB^N_s = \lim_{N\to\infty} 
\hat B\Big(\int_0^t h(I^N_s)ds\Big) = \hat B\Big(\int_0^t h(I_s)ds\Big).$$
From this and \eqref{eq:s3}, we can thus write
\begin{align*}
    \mathbb E\Big[\int_0^t \sqrt{h(I_s)} dB_s \int_0^t \sqrt{h(I_s)} d\widetilde B_s\Big] & = \mathbb E\Big[\hat B\Big(\int_0^t h(I_s) ds\Big) \cdot B'\Big( \int_0^t h(I_s) ds\Big)\Big] \\ & = (2p-1)\int_0^t h(I_s) ds
\end{align*}
and we are done.
\end{proof}

\begin{proof}[Proof of Corollary \ref{cor2}]
Actually, the proof of \eqref{eq:cor11}, \eqref{eq:cor12} and \eqref{eq:cor13} is literally the same as in the proof above. 
\end{proof}

\end{document}